\documentclass{amsart}

\usepackage{enumerate}
\usepackage{amsmath, amscd}
\usepackage{ifpdf}
\usepackage{amsfonts}
\usepackage[normalem]{ulem}
\usepackage{color}
\usepackage{amssymb}
\usepackage{amsthm}
\usepackage[hyperfootnotes=false]{hyperref}
\usepackage{setspace}
\usepackage{graphicx}

\newtheorem{thm}{Theorem}
\newtheorem{main theorem}[thm]{Main Theorem}
\newtheorem{corollary}[thm]{Corollary}

\newtheorem{lemma}[thm]{Lemma}
\newtheorem{prop}[thm]{Proposition}
\newtheorem*{prop*}{Proposition}

\theoremstyle{definition}

\newtheorem{remark}[thm]{Remark}

\newcommand{\bea}{\begin{eqnarray*}}
\newcommand{\eea}{\end{eqnarray*}}

\newcommand{\be}{\begin{equation}}
\newcommand{\ee}{\end{equation}}

\ifpdf
\title[Attracting skew products]{Fatou components of attracting skew products}

\author[Han Peters, Iris M. Smit]{Han Peters, Iris Marjan Smit}

\address{KdV Institute for Mathematics\\
University of Amsterdam\\
The Netherlands}
\email{h.peters@uva.nl}
\email{imsmit@gmail.com}

\begin{document}
\bibliographystyle{plain}

\begin{abstract}
We investigate the existence of wandering Fatou components for polynomial skew-products in two complex variables. In 2004 the non-existence of wandering domains near a
super-attracting invariant fiber was shown in \cite{Lilov}. In 2014 it was shown in \cite{ABDPR} that wandering domains can exist near a parabolic invariant fiber. In
\cite{PV2014} the geometrically-attracting case was studied, and we continue this study here. We prove the non-existence of wandering domains for subhyperbolic attracting skew products; this class contains the maps studied in \cite{PV2014}. Using expansion properties on the Julia set in the invariant fiber, we prove bounds on the rate of escape of critical orbits in almost all fibers. Our main tool in describing these critical orbits is a possibly singular linearization map of unstable manifolds.
\end{abstract}

\maketitle

\tableofcontents

\section{Introduction}

Sullivan's No Wandering Domains Theorem \cite{Sullivan} states
that every Fatou component of a rational function $f\colon \hat{\mathbb
C} \rightarrow \hat{\mathbb C}$ is either periodic or
pre-periodic. It was recently shown in \cite{ABDPR} by Astorg,
Buff, Dujardin, Raissy and the first author that Sullivan's
Theorem does not hold for polynomial maps in $\mathbb C^2$. The
maps constructed in \cite{ABDPR} are polynomial skew-products of
the form
$$
(z,w) \mapsto \bigl(f(z,w), g(w)\bigr),
$$
where $f(z,w) = p(z) + q(w)$ and both polynomials $p$ and $g$ have
a parabolic fixed point at the origin. The wandering Fatou
components are contained in the parabolic basin of $g$.

In the current project we investigate whether polynomial
skew-products can also have wandering domains in the basin of an
attracting fixed point of $g$. Throughout the paper, we will
assume that the first coordinate function $f$ of the skew-product
has the form $f(z,w) = z^d + a_{d-1}(w) z^{d-1} + \ldots + a_0(w)$.
Recall that in the paper on polynomial skew-products by Jonsson
\cite{Jonsson} a polynomial skew-product is assumed to extend
holomorphically to $\mathbb P^2$. We do not make this stronger
assumption here. We will prove the following.

\begin{thm}\label{thm:main}
Let $F\colon(z,w)\mapsto (f(z,w),g(w))$ be a polynomial skew product and assume that $0 = g(0)$ is an attracting fixed point with corresponding basin $B_g$. Further assume that the polynomial $p(z) = f(z,0)$ is subhyperbolic. Then $F$ has no wandering Fatou components contained in $B_g$.
\end{thm}

Note that any Fatou component of a skew-product $F$ must be
contained in a Fatou component of the polynomial $g$. Hence if $F$
has a wandering Fatou component, then by Sullivan's Theorem there
must be a wandering Fatou component that is contained in a
periodic Fatou component of $g$. This component must be either
attracting, parabolic or a Siegel disk. As mentioned before, the
existence of wandering domains in parabolic basins of $g$ was
shown in \cite{ABDPR}. In \cite{Lilov}, Lilov showed that there can
be no wandering Fatou components in a super-attracting basin of
$g$. In this work we will focus on the geometrically attracting
case.

\medskip

Without loss of generality we may assume that $0 = g(0)$ is an attracting fixed point. Note that if we can rule out the existence of wandering domains in a small neighborhood of
$\{w=0\}$, then it follows that there are no wandering domains in the entire basin of attraction of $0$. When working in a small neighborhood of $\{w=0\}$ we can change
coordinates so that $g(w) = \lambda w$, with $0< |\lambda| < 1$. A consequence is that the coefficients $a_0, \ldots a_{d-1}$ of $f$ are generally no longer polynomials, but only
depend holomorphically on $w$. This will not be an issue for us, and in fact our results hold for $g$ holomorphic as well. For simplicity of notation we will continue to assume
that $F(z,w) = (f(z,w) , \lambda w)$ is a polynomial skew-product.

\medskip

The geometrically attracting case was recently studied by Vivas and the first author in \cite{PV2014}. In order to state results from that paper, let us first recall in greater
detail what Lilov proved for the super-attracting case. Lilov first showed that every $1$-dimensional Fatou component of $p(\cdot) = f(\cdot,0)$ is contained in a Fatou
component of $F$, which we will refer to as a \emph{bulging} Fatou component of $p$ (by slight abuse of language). Then, Lilov showed that the forward orbit of any nearby horizontal disk must intersect a
bulging Fatou component of $p$. By horizontal disk we mean a disk contained in a fiber $\{w = w_0\}$, and by nearby we mean that $w_0$ is contained in the basin of $\{w = 0\}$.
The non-existence of wandering Fatou components in a super-attracting basin of $g$ follows immediately.

\medskip

The bulging of Fatou components holds in the geometrically attracting case as well. However, in \cite{PV2014} explicit attracting polynomial skew-products were constructed for
which there are nearby horizontal disks whose forward orbits avoid the bulging Fatou components of $p$. What we show in the current work is that in \emph{most} fibers such disks
cannot exist.

\begin{prop}\label{prop:main}
Consider a skew-product
$$
F(z,w) = \bigl(f(z,w), \lambda w\bigr)
$$
satisfying the conditions of Theorem \ref{thm:main}, with
$0<|\lambda|<1$. Then there is a set $E \subset \mathbb{C}$ of
full measure, such that for every $w_0 \in E$ the
forward orbit of every disk in the fiber $\{w = w_0\}$ must
intersect a bulging Fatou component of $p$.
\end{prop}

In particular the set $E$ is dense, so Theorem \ref{thm:main} follows immediately.
Another stronger corollary is that the only Fatou components of
$F$ are the bulging Fatou components of $p$. This follows
immediately from the fact that the topological degree of $F$
equals the degree of $p$, hence the only Fatou components that can
be mapped onto the bulging Fatou components of $p$ are those
bulging Fatou components.

\medskip

The skew-products constructed in \cite{PV2014} are of the form
$$
F(z,w) = \bigl(p(z) + q(w), \lambda w\bigr),
$$
where $p$ has only a single critical point which lies in the Julia set and is pre-periodic. Hence the maps from \cite{PV2014} satisfy the conditions in Theorem \ref{thm:main}.

\medskip

Let us emphasize here that the proof of Sullivan's No Wandering Domains Theorem relies in an essential way on the Mapping Theorem for quasiconformal maps. There is no known analogue for the Mapping theorem that can be used to prove the non-existence of wandering domains in higher dimensions. In light of the construction of wandering domains in \cite{ABDPR}, it seems unlikely that this line of argument can be used even for polynomial skew-products. For general one-dimensional polynomial or rational functions there is no alternative for the use of quasiconformal deformations to prove the non-existence of wandering domains.

However, under additional assumptions on the post-critical set there do exist alternative proofs. The easiest class is given by the hyperbolic polynomials, for which the forward orbits of all critical points stay bounded away from the Julia set. In this case a sufficiently large iterate of the map is expanding on the Julia set, which immediately implies the non-existence of wandering Fatou components. It was pointed out by Lilov in \cite{Lilov} that an attracting or super-attracting skew-product acting hyperbolically on the invariant fiber has no wandering Fatou components. In this article we consider skew-products satisfying a weaker assumption, namely that the polynomial acting on the invariant fiber is subhyperbolic.

\medskip

In the subhyperbolic case, sometimes referred to as Misiurewicz, the polynomial is assumed to have no parabolic periodic points, and the critical points that lie on the Julia set are pre-periodic. In fact, these assumptions can be taken as the definition of subhyperbolicity, an alternative definition being that the polynomial is expanding with respect to a so-called admissible metric.

Let us recall a classical argument due to Fatou to show that a subhyperbolic polynomial $p$ does not have wandering Fatou components. Imagine that it does, and let $K$ be a compact subset of such a wandering Fatou component. Then for some subsequence $(n_j)$ the sets $f^{n_j}(K)$ must converge to a point $x \in \mathcal{J}_p$ that does not lie in the post-critical set. Let $U$ be a disk centered at $x$ that does not intersect the post-critical set. Then on $U$ all inverse branches of every iterate $p^n$ are conformal, and necessarily map $U$ inside the disk of radius $R$, the escape radius. It follows that as the Poincar\'e diameter of $f^{n_j}(K)$ in $U$ converges to $0$, the Poincar\'e diameter of $K$ in $D_0(R)$ must also converge to $0$, which gives a contradiction.

\medskip

We will follow the same line of argument in this paper, the main difference being that instead of having no critical points enter the disk $U$, we obtain sub-linear estimates on the degree of the inverse branches. This will turn out to be sufficient to conclude the non-existence of wandering domains.

\medskip

The outline of the proof of Proposition \ref{prop:main} is discussed in the next section. The details are covered in Sections \ref{section: Linearization maps} through
\ref{section:areas}. In Section \ref{section:resonant} we discuss the proof under stronger assumptions, close to the assumptions used in \cite{PV2014}, in which case a much shorter argument is obtained.

\section{Outline and background} \label{section:outline}

Throughout the rest of this paper, we work with a map
$$
F(z,w)=\bigl(f(z,w),\lambda w\bigr),
$$
where $|\lambda|<1$, $f(z,w)=z^d+a_{d-1}(w)z^{d-1}+\ldots +a_0(w)$ and
$p(z)=f(z,0)$. We will write $F^n$ for the $n$-th iterate
of $F$, and we will use the convenient notation
$$
F^n(z,w) = \bigl(F^n_1(z,w), F^n_2(z,w)\bigr).
$$

We assume that the polynomial $p$ is subhyperbolic. A consequence is that the Fatou set consists of a finite collection of attracting basins, and that the orbit of each critical point in the Fatou set converges to one of the attracting cycles. Every critical point in $\mathcal{J}_p$ is eventually mapped onto a repelling periodic point. Passing to some iterate of $F$ and $p$, we may assume that all our critical points in $\mathcal{J}_p$ are not just preperiodic, but will eventually be mapped onto repelling fixed points.

\medskip

A subhyperbolic polynomial is in particular also
\emph{semi-hyperbolic}. Recall that a polynomial $p$ is called
semi-hyperbolic if the Julia set does not contain parabolic
periodic points or recurrent critical points. It is an interesting
question whether the methods introduced in this article can be
used to prove Theorem \ref{thm:main} under the more general assumption that $p$ is
semi-hyperbolic.

\medskip

The main idea in the proof of Proposition \ref{prop:main} is that the fibers $\{w = w_0\}$ in $E$ are chosen such that the critical points in the fibers $\{w = \lambda^n w_0\}$ escape a given neighborhood of the Julia set sufficiently fast under iterations of $F$. To be more precise, fix $R>0$ large enough
such that
$$
|z| > R \; \; \mathrm{implies \; that} \; \; |p(z)| > 2|z|.
$$
For each attracting periodic point $y$ of $p$, fix an open neighborhood $W_y$ of its orbit such that $\overline{p(W_y)} \subset W_y$.
Define
$$
W_0=\{|z|>R\} \cup \bigcup_{y}W_y,
$$
where we take the union over all attracting periodic points $y \in \mathbb{C}$.

Since $p$ has only finitely many attracting periodic orbits, we can find $\epsilon>0$ be such that the
set
$$
W:=W_0 \times D(0,\epsilon)\subseteq \mathbb{C}^2
$$
satisfies $\overline{F(W)} \subseteq W$. Fix a constant $M$ such that $|DF| \leq M$ on
$D(0,R+1) \times D(0,\epsilon)$.

Define the sets $U_m=p^{-m}(W_0)$ and write $V_m$ for their complements. By slight abuse of
notation we consider these as subsets of $\mathbb{C}_z =
\mathbb{C} \times \{0\} \subset \mathbb C^2$ as well. Notice that
$$
\mathcal{F}_p = \bigcup_{m \in \mathbb{N}} U_m.
$$
By definition, points in $U_m$ must escape to $W_0$
in at most $m$ steps. An
elementary computation shows that a similar statement holds for points
sufficiently close to $U_m$:

\begin{lemma}\label{lemma:Cirkelverlaten}
There exist $C_1>0$ such that for all $m
\in \mathbb N$ the following holds:
$$
\text{If }\textrm{d}\bigl((z,w),U_m \bigr) < C_1 M^{-m}, \quad
\text{then }F^{m+1}(z,w) \in W.
$$
\end{lemma}

Since $p$ is subhyperbolic, it is also semi-hyperbolic. These polynomials were studied by
Carleson, Jones and Yoccoz in \cite{CJY1994}. In Section
\ref{section: John Domains} we will use their estimates to prove the following:

\medskip

\noindent{\bf Proposition \ref{prop:area}.} \emph{ The area of the
sets $V_m$ decreases exponentially with $m$. }

\medskip

We note that this estimate does not hold for general polynomials;
for example, it does not hold for polynomials with a parabolic
periodic point. As the exponential decay of these areas will play
a crucial role in our proof, it is clear that one should not
expect our proof to work for polynomials $p$ that are not
semi-hyperbolic.

\medskip

Let $U\subseteq \mathbb{C}$ be any neighborhood of the
post-critical set of $p$. Using Proposition \ref{prop:area} we
will be able to choose fibers $\{w = w_0\}$ for which all critical
points in the fibers $\{w = \lambda^n w_0\}$ escape to $W$,
with bounds on the number of steps it takes for the orbit of such
a critical point to leave $U \times \mathbb{C}$ and land in $W$.
As a consequence we will obtain the following:

\medskip

\noindent {\bf Proposition \ref{prop:final critical pointcount}.} \emph{
There exists a set $E \subset \mathbb{C}$, of full measure in some neighborhood of the origin, with the following property:
For all $w \in E$ there exists a constant $C_2 = C_2(w,U)$ such that for all $n \in \mathbb{N}$ we have
$$
\#\left\{z:~ \frac{\partial F_1^n}{\partial z}(z,w)=0 \mathrm{\;
and \;} F_1^n(z,w) \notin W_0 \cup U  \right\} \leq C_2
\sqrt{n}.
$$
}

\medskip

Recall that each critical point $x \in \mathcal{J}_p$ of $p$ is eventually mapped to a
repelling fixed point $x_r$, with multiplier $\mu$ where $|\mu|>1$. The main tool
for controlling the orbits of the critical points of the
polynomial skew product is a linearization map of the unstable
manifold of this repelling fixed point, given by a map $\Phi\colon
\mathbb C \rightarrow \mathbb C$ that satisfies $\Phi(\mu t) = p^k
\circ \Phi(t)$ for some $k \in \mathbb{N}$. The construction of these linearization maps will
be discussed in Section \ref{section: Linearization maps}.

\medskip

The bound on the number of critical points provided by Proposition \ref{prop:final critical pointcount} gives us degree estimates for some proper holomorphic maps between
hyperbolic Riemann surfaces, which in turn can be used to obtain area-estimates, using the following proposition, which will be proved in Section \ref{section:areas}.

\medskip

\noindent {\bf Proposition \ref{prop:area inv}.} \emph{ There
exists a uniform constant $C_3>0$ such that the following holds:
Let $f\colon\mathbb D \rightarrow \mathbb D$ be a proper holomorphic
map of degree $d$, and let $R \subseteq \mathbb{D}$ have
Poincar\'{e} area $A$. If $d\cdot A^{1/2d}<1/8$, then the
Poincar\'e area of $f^{-1}(R)$ will be at most $C_3 d^3 A^{1/d}$.
}

\medskip

With these ingredients we are ready to prove our main results.

\medskip

\noindent {\bf Proof of Proposition \ref{prop:main}.} Let $E$ be
as in Proposition \ref{prop:final critical pointcount}, and
suppose for the purpose of a contradiction that a fiber $\{w =
w_0\}$, where $w_0 \in E$, does contain a disk $D$ whose forward
orbit avoids the bulging Fatou components of $p$. Then the
restriction of $F^n$ to $D$ is bounded and hence a normal family.
Therefore, there exists a subsequence $F^{n_j}$ such that
$F^{n_j}|_D$ converges, uniformly on compact subsets of $D$, to a
holomorphic map to $\mathcal{J}_p$. As $\mathcal{J}_p$ has no
interior, $F^{n_j}(D)$ must in fact converge to a point $\zeta \in
\mathcal{J}_p$. After shrinking $D$ if necessary, we may assume
that the convergence of $F^{n_j}(D)$ to $\zeta$ is uniform.

\medskip

In particular, as $F(W) \subseteq W$, we know that $F^n(D)$ will
never intersect $W$. It therefore follows from Lemma
\ref{lemma:Cirkelverlaten} that whenever $|\lambda^n w_0|  < C_1
M^{-m}$, we have $F_1^n(D) \subseteq V_m$. By Proposition
\ref{prop:area}, we know that $\textrm{Area}(V_m)$ decreases
exponentially. Hence there exist constants $C > 1$, and $\gamma <
1$ such that
\begin{equation}\label{eq:area F^n(D)}
\mathrm{Area}\bigl(F^n(D)\bigr) \leq C \gamma^n.
\end{equation}

By our assumption that the critical points in $\mathcal{J}_p$ are
all eventually mapped onto repelling fixed points, we may
assume that $\zeta$ does not lie in the post-critical set. So
choose our open neighbourhood $U \subseteq \mathbb{C}_z$ of the
post-critical set in such a way that $\zeta \notin U$, and let
$r>0$ be such that $D(\zeta,r) \cap (U \cup W_0)= \emptyset$. Let $J_1 \in
\mathbb N$ be such that $F_1^{n_j}(D) \subseteq
D(\zeta,\frac{r}{2})$ for all $j \geq J_1$. Define $O_j$ to be the
connected component of $\left( F^{n_j}
\right)^{-1}\left(D(\zeta,r)\times \{\lambda^{n_j} w_0\}\right)$
that contains $D$. Then we have $D \subseteq O_j \subseteq
D(0,R)\times \{w_0\}$, and we can regard $F_1^{n_j} \colon O_j
\rightarrow D(\zeta,r)$ as a one-dimensional proper holomorphic
map.

\medskip

Proposition \ref{prop:final critical pointcount} tells us that
this map has at most $d_j=C_2\sqrt{n_j}$ critical points. We are
now in position to apply Proposition \ref{prop:area inv}. Set
$R_j=F_1^{n_j}(D)$, with Poincar\'{e} area
$A_j=\text{Area}_{D(\zeta,r)}(R_j)$ with respect to $D(\zeta,r)$.
As $j \geq J_1$, we have $R_j\subseteq D(\zeta,\frac{r}{2})$.
Hence we can estimate $A_j$ using our result on the Euclidean area
of $R_j$ as found in Equation \eqref{eq:area F^n(D)}:
$A_j<C'\gamma^{n_j}$ for some uniform constant $C'$. Then
$$
d_j A_j^{1/2d_j}< C_2 \sqrt{n_j} \left( C' \gamma^{n_j}
\right)^{1/2C_2 \sqrt{n_j}}<C_2(C')^{1/2C_2 \sqrt{n_j}} \sqrt{n_j}  \gamma^{\sqrt{n_j}/2C_2}.
$$
Since $\gamma<1$, this expression converges to zero as $j$ increases. Therefore we can find $J_2>J_1$ such
that $d_j A_j^{1/2d_j}<1/8$ for all $j>J_2$.

In this setting, Proposition \ref{prop:area inv} implies that
$$
\text{Area}_{D(0,R)}(D) \leq \text{Area}_{O_j}(D) \leq C_3 d_j^3 A_j^{1/d_j} < C_3C_2^3 n_j^{3/2} \left(C' \gamma^{n_j} \right)^{1/C_2 \sqrt{n_j}}.
$$
However, this last expression will also tend to zero as $j$
increases, which gives a contradiction. \hfill $\square$

\medskip

For the rest of this paper, our goal is to prove Propositions \ref{prop:area}, \ref{prop:final critical pointcount} and \ref{prop:area inv}. First, we define the linearization
map $\Phi$ that is used to describe the behavior of the critical points in the Julia set. In Section \ref{section:resonant} we prove our results for the special class of maps studied by Vivas and the first author in \cite{PV2014}. In this setting the proof is much shorter, and we will give a more precise description of the post-critical orbits.

In Section \ref{section: John Domains}, we get back to the general case and prove Proposition \ref{prop:area}. This estimate will be used to prove Proposition \ref{prop:final
critical pointcount} in the next section. And finally, we prove Proposition \ref{prop:area inv} in Section \ref{section:areas}.

\section{Linearization maps}\label{section: Linearization maps}

In this section we construct the linearization map $\Phi$ that
will later help us track the orbits of the critical points of $F$.
As $F(z,w)=\left( f(z,w), \lambda w\right)$ with $f(z,w)=
z^d+a_{d-1}(w)z^{d-1}+ \cdots + a_0(w)$, the critical points of
$F$ are those points where $\frac{\partial f(z,w)}{\partial z}=0$.

In a small neighborhood $\{|w|<\epsilon \}$ of the
$\{w=0\}$-fiber, these critical points will form finitely many,
possibly singular, varieties $K_1$, \ldots, $K_q$, each of which
intersects $\{w=0\}$ in a single point. Let $K$ be such a variety, with intersection $(x_0,0)$ with $\{w=0\}$. If $x_0 \in \mathcal{J}_p$, then $x_0$ must be eventually mapped to a repelling fixed point $x_r=p^r(x_0)$ of $p$, which is a saddle point for $F$. The variety $F^r(K)$ will pass through this saddle point. The linearization results
in this section will help us study the orbits originating on these varieties.

\bigskip

Let $G\colon (\mathbb C^2,0) \rightarrow (\mathbb C^2,0)$ be a
holomorphic germ, and assume that $0$ is a saddle fixed point.
Without loss of generality we may assume that the stable direction
is $(0,1)$ and the unstable direction is $(1,0)$, so that $G$ is
of the form
$$
G\colon(z,w) \mapsto \bigl(\mu z + h.o.t., \lambda w + h.o.t.\bigr),
$$
with $|\mu|>1$ and $|\lambda| < 1$. After changing coordinates we
may assume that, possibly on a smaller neighborhood of the origin,
the stable and unstable manifolds are equal to respectively
$\mathbb C_w=\{0\} \times \mathbb{C}$ and $\mathbb C_z=\mathbb{C}
\times \{0\}$.

\medskip

Let $V$ be an irreducible analytic variety through the origin. We
can locally write $V$ as
$$
V = \bigl\{\left(\gamma_1(t), \gamma_2(t)\right)\bigr\},
$$
where
$$
\gamma_1(t) = a t^k + h.o.t. \; \; \; \mathrm{and} \; \; \;
\gamma_2(t) = b t^l + h.o.t..
$$

\begin{remark} When $\gamma_1=0$, the variety $V$ coincides with
the stable variety $\{z=0\}$ of $G$. In this case,
$(G^n(\gamma(t)))_{n \in \mathbb{N}}$ will converge to the origin
along the line $\{z=0\}$, uniformly on compact subsets.
\end{remark}

When $\gamma_1 \neq 0$, the order $k$ of the parametrization is well-defined. We will
prove the following proposition:

\begin{prop}\label{prop:general linearization}
The maps $\Phi_n$ defined by
$$
\Phi_n(t) = G^{kn}\left(\gamma_1(\mu^{-n}t),
\gamma_2(\mu^{-n}t)\right)
$$
converge uniformly on compact subsets of $\mathbb C$ to a
holomorphic map
$$
\Phi\colon \mathbb C \rightarrow \mathbb C_z
$$
of local order $k$ that satisfies the functional equation
$G^k\circ \Phi(t) = \Phi(\mu t)$. This convergence is
exponentially fast on compact subsets.
\end{prop}

\begin{proof}
Let us start with the case where $\gamma_2=0$, as it is more
straightforward. We use the following variation on Koenigs'
Theorem.

\begin{lemma}\label{lemma:Koenigs}
Let
$$
g(z) = \mu^k z + h.o.t.
$$
and
$$
h(t) =a t^k + h.o.t.
$$
be holomorphic functions defined in a neighborhood of the origin,
with $k \ge 1$, $a \neq 0$ and $|\mu| > 1$. Then the sequence of maps
$$
\varphi_n(t) = g^n \circ h(\mu^{-n} t)
$$
converges uniformly and exponentially fast on compact subsets of
$\mathbb C$.
\end{lemma}

\begin{proof}
Using Koenigs' Theorem we can write
$$
g(z) = \psi^{-1}\left(\mu^k \cdot \psi(z)\right),
$$
for a germ $\psi$ of the form
$$
\psi(z) = z + h.o.t. .
$$
Hence $g^n(z) = \psi^{-1} ( \mu^{kn} \psi(z))$, and
$$
g^n \circ h(\mu^{-n} t) = \psi^{-1}\left(\mu^{kn} \psi\circ
h(\mu^{-n}t)\right).
$$
Writing
$$
\psi \circ h(t) = a t^k + t^{k+1} \cdot \eta(t)
$$
we obtain
$$
\mu^{k(n+1)} \psi \circ h (\mu^{-(n+1)}t) - \mu^{kn} \psi \circ h
(\mu^{-n}t) =
\mu^{-n}t^{k+1}\left[\mu^{-1}\eta(\mu^{-(n+1)}t)-\eta(\mu^{-n}t)\right].
$$
The statement follows.
\end{proof}

We continue the proof of Proposition \ref{prop:general linearization}. Use the notation
$$
G^k(z,w)=\left(G_1^k(z,w),G_2^k(z,w)\right)
$$
for the components of the $k$-th iterate of $G$. Since $G$ leaves both axes invariant, we can write
$$G_1^k(z,w)=\mu^k z + z \cdot f_1(z,w) \quad \text{ and } \quad G_2^k(z,w)=\lambda^k w
+ w \cdot f_2(z,w)$$
where $f_1(0,0) = f_2(0,0) = 0$.

By our assumption that $\gamma_2=0$, it follows that $\Phi_n(t) =
\left(g^{n}(\gamma_1(\mu^{-n} t)),0\right)$ for all $n \in \mathbb
N$, where $g(\cdot)=G_1^k(\cdot,0)$. By Lemma \ref{lemma:Koenigs}
the sequence $g^{n}(\gamma_1(\mu^{-n} t))$ converges uniformly and
exponentially fast on compact subsets of $\mathbb C$.

\medskip

Now that we have dealt with the cases where $\gamma_1=0$ or
$\gamma_2=0$, we tackle the general case. Write $\gamma_1(t)=a
t^k + t^{k+1} \tilde{\gamma}_1(t)$ with $a \neq 0$, and likewise
$\gamma_2(t)=b t^l + t^{l+1} \tilde{\gamma}_2(t)$ with $b \neq
0$.

Define
$$
\Phi_n(\zeta, \xi) = G^{kn}\left(\gamma_1(\mu^{-n}\zeta),
\gamma_2(\mu^{-n}\xi)\right),
$$
as a function of two distinct variables. Our goal is to show that
for $\zeta, \xi$ small we have $\Phi_{n}(\zeta, \xi) =
\Phi_{n-1}(\zeta^\prime, \xi^\prime)$, where $\zeta^\prime \sim
\zeta$ and $|\xi^\prime| << |\xi|$.

We start with the trivial equation
$$ \Phi_{n}(\zeta, \xi)  = G^{k(n-1)} \circ G^k \left(\gamma_1(\mu^{-n}\zeta),
\gamma_2(\mu^{-n}\xi)\right).$$ To unravel this, note that for
$\zeta, \xi \in D(0,\epsilon)$ with $\epsilon$ sufficiently small
(independent of $n$), we have
\begin{align*}
 G_1^k&\left(\gamma_1(\mu^{-n}\zeta),\gamma_2(\mu^{-n}\xi)\right)\\
 &= a \mu^{-k(n-1)}\zeta^k + \mu^{-k(n-1)-n}
 \zeta^{k+1}\tilde{\gamma}_1(\mu^{-n}\zeta)
 +\gamma_1(\mu^{-n}\zeta)f_1\left(\gamma_1(\mu^{-n}
 \zeta),\gamma_2(\mu^{-n}\xi)\right)\\
 &= a (\mu^{-(n-1)} \zeta)^k \biggl(1+ \mu^{-n} a^{-1}
 \zeta \tilde{\gamma}_1(\mu^{-n}\zeta) \\
 &\qquad \qquad \qquad \qquad \qquad +\mu^{-n}a^{-1}\mu^{-k}
 \frac{\gamma_1(\mu^{-n}\zeta)}{(\mu^{-n}\zeta)^k}\frac{f_1\left(\gamma_1(\mu^{-n}
 \zeta),\gamma_2(\mu^{-n}\xi)\right)}{\mu^{-n}}\biggr)\\
 &=\gamma_1\left(\mu^{-(n-1)}
 \psi_{n}(\zeta,\xi)\right).
 \end{align*}

Here, the function $\psi_{n}$ has the form
$\psi_{n}(\zeta,\xi)=\zeta\bigl(1+\mu^{-n}
\tilde{\psi}_{n}(\zeta,\xi)\bigr)$ with $\tilde{\psi}_{n}$ holomorphic
and bounded on $D(0,\epsilon) \times D(0,\epsilon)$, uniformly in
$n$.

Likewise, we can work out that
\begin{align*}
G_2^k&\left(\gamma_1(\mu^{-n}\zeta),\gamma_2(\mu^{-n}\xi)\right)\\
&=\lambda^k b \mu^{-ln} \xi^l +\lambda^k
\mu^{-nl-n}\xi^{l+1}\tilde{\gamma}_2(\mu^{-n}\xi)
+\gamma_2(\mu^{-n}\xi)f_2\left(\gamma_1(\mu^{-n}
\zeta),\gamma_2(\mu^{-n}\xi)\right)\\
&=b_k (\lambda^{k/l}\mu^{-n}\xi)^l
\biggl(1+ \mu^{-n}b^{-1}\xi \tilde{\gamma}_2(\mu^{-n}\xi)\\
& \qquad \qquad \qquad \qquad \qquad +
\mu^{-n}b_k^{-1}\lambda^{-k} \frac{
\gamma_2(\mu^{-n}\xi)}{(\mu^{-n}\xi)^l}\frac{f_2\left(\gamma_1(\mu^{-n}\zeta),\gamma_2(\mu^{-n}\xi)\right)}{\mu^{-n}}\biggr)\\
&=\gamma_2\left(\mu^{-(n-1)}\chi_n(\zeta,\xi)\right).
\end{align*}

The function $\chi_n$ has the shape
$\chi_n(\zeta,\xi)=\lambda^{k/l}\mu^{-1} \xi \bigl(1+\mu^{-n}
\tilde{\chi}_n(\zeta,\xi)\bigr)$, once again with
$\tilde{\chi}_{n}$ bounded on $D(0,\epsilon) \times
D(0,\epsilon)$, uniformly in $n$.

Writing $\varphi_n=(\psi_n,\chi_n)$, we now have
$$
\Phi_n(\zeta,\xi)=\Phi_{n-1} \circ \varphi_n(\zeta,\xi).
$$

Let $C$ be an upper bound for all $|\tilde{\psi}_n|$ and
$|\tilde{\chi}_n|$ on $D(0,\epsilon)\times D(0,\epsilon)$, and
choose $N$ large enough to ensure that $\left|\frac{\mu^{-N}
C}{1-|\mu^{-1}|} \right| < \epsilon /2$ and
$|\lambda^{k/l}\mu^{-1} |\left(1+|\mu|^{-N} C \right)<1$. Then a
straightforward induction on $m$ shows that
$$
\varphi_{n}\circ \cdots \circ \varphi_{n+m}(\zeta, \xi ) \in
D\left(0,\frac{\epsilon}{2}+\frac{
|\mu|^{-n}C}{1-|\mu|^{-1}}\right) \times D(0, \epsilon)
$$
for all $n \geq N$, $m \in \mathbb{N}$ and $(\zeta,\xi) \in
D(0,\epsilon/2)\times D(0,\epsilon)$.

We can now write $\Phi_n(\zeta,\xi)=\Phi_N\left( \varphi_{N+1}
\circ \cdots \circ \varphi_n(\zeta,\xi)\right)$. By our bounds on
$|\tilde{\psi}_n|$ and $|\tilde{\chi}_n|$, the sequence
$(\Phi_n(\zeta,\xi))_{n \geq N}$ will converge uniformly and
exponentially fast on $D(0,\epsilon/2)\times D(0,\epsilon)$.

Plugging in $\zeta=\xi=t$ now proves the proposition:
$(\Phi_n(t))_{n \in \mathbb{N}}$ must converge on
$D(0,\epsilon/2)$. By construction, it must satisfy the functional
equation
$$
G^k \circ \Phi(t)=\Phi(\mu t).
$$
This equation also gives us global convergence of $(\Phi_n(t))_{n
\in \mathbb{N}}$ and uniform exponentially fast convergence on
compact subsets. By construction, $\Phi$ must have local order
$k$.
\end{proof}

\begin{remark}\label{remark: linearization F}
Let us say a few words about the role that the linearization map $\Phi$ will play in the rest of this paper. We have a polynomial skew-product $F$ of the form
$$
F(z,w) = \bigl(f(z,w), \lambda \cdot w\bigr),
$$
where $f(z,w)=z^d+a_{d-1}(w)z^{d-1}+\ldots+a_0(w)$, and we write
$p(\cdot) = f(\cdot,0)$. As mentioned at the start of this
section, in some small neighborhood $\{|w|<\epsilon\}$ of the
$\{w=0\}$-fiber, the critical points of $F$ form finitely many
irreducible critical varieties $K_1, \ldots K_q$. We may assume,
shrinking $\epsilon$ if necessary, that each $K_i$ intersects
$\{w=0\}$ at one unique point, which must be a critical point of
$p$. Assume for now that this critical point lies in $\mathcal{J}_p$.

We study these varieties one at a time. Write $K$ for our
irreducible critical variety, and $x_0$ for its intersection with
$\{w=0\}$. We can locally parameterize $K$ by
$$
K = \bigl\{(\gamma_1(t), \gamma_2(t))\bigr\}.
$$

By our assumptions on $p$ each of its critical points in
$\mathcal{J}_p$ is eventually mapped to a repelling fixed point $p^r(x_0)
= x_r$ of $p$. Write $\mu=p'(x_r)$.

Then $F^r(K)$ is a variety though $(x_r,0)$, where $F$ has a
saddle fixed point. Let $\Psi$ be a local change of coordinates
such that $G=\Psi \circ F \circ \Psi^{-1}$ satisfies the
conditions of Proposition \ref{prop:general linearization}: its
stable and unstable manifolds are $\mathbb{C}_w$ and
$\mathbb{C}_z$ respectively. We can write $G(z,w)$ as
$$
G \colon (z,w) \mapsto \bigl(\mu z+zg_1(z,w),\lambda w+wg_2(z,w)
\bigr).
$$
Translating to our new coordinates, our variety of interest
becomes $\Psi\left(F^r(K)\right)$. It can locally be parameterized
by
$$
\bigl\{\Psi \left( F^r\left( \gamma(t)\right)\right)\bigr\}.
$$
If it is equal to $\mathbb{C}_w$, the original variety $F^r(K)$
must lie in the stable manifold $\Sigma_F^s(x_r)$. Otherwise, the
first coordinate of $\Psi \left( F^r\left(
\gamma(t)\right)\right)$ can be written as $at^k + h.o.t.$ where
$a \neq 0$ and $k \geq 1$.

Note that $\Psi \left( F^{kr} \left( \gamma(t) \right)
\right)=G^{kr-r}\circ \Psi \left( F^r\left(
\gamma(t)\right)\right)$ has $\mu^{kr-r}a t^k + h.o.t.$ as its
first coordinate; the value $k$ did not change. We can therefore
apply Proposition \ref{prop:general linearization} to find that
the maps
\begin{align*}
\Phi_n(t)&=F^{kn}\left(\gamma(\mu^{-n}t)\right)\\
&=\Psi^{-1} \circ G^{k(n-r)}\left(\Psi \circ F^{kr} \circ
\gamma\left(\mu^{-r} \mu^{-(n-r)}t\right) \right)
\end{align*}
must converge, uniformly and exponentially fast on compact
subsets, to a holomorphic map $\Phi\colon \mathbb{C} \rightarrow
\Psi^{-1}(\mathbb{C}_z)=\Sigma_F^u(x_r)=\mathbb{C}_z$ of local order $k$. By construction,
$\Phi$ must satisfy the functional equation $F^k \circ \Phi(t) = \Phi(\mu t)$.
This map will be used all through Sections \ref{section:resonant} and \ref{section:tracking}
to study the orbits of critical points in varieties $K$ that pass through the Julia set of $p$.
\end{remark}

\section{The resonant case}\label{section:resonant}

In this section we consider the special resonant case, which was
also studied by Vivas and the first author in \cite{PV2014}. In
this case it is considerably easier to prove the non-existence of
wandering domains, and we will obtain a more precise description
of the critical orbits. Although the results in this section are
not used in the proofs of Proposition \ref{prop:main} and Theorem
\ref{thm:main}, they did serve as an inspiration.

\medskip

We consider polynomial skew-products of the form
$$
F(z,w) = \bigl(p(z) + q(w), \lambda \cdot w\bigr),
$$
for which we now assume that the polynomial $p$ has only a single
critical point $x_0$, which in a finite number of steps is mapped
onto a repelling fixed point $x_r = p^r(x_0)$. Without loss of
generality we may assume that $x_r = 0$. As usual we write
$p^\prime(0) = \mu$, which satisfies $|\mu|>1$. To keep the study
case in this section as straightforward as possible, we assume
that the critical variety $K=\{(x_0,t)\}$ has an image $F^r(K)$
that is transverse to the stable variety $\Sigma_F^s(0)$. We
further assume the following resonant condition:
$$
\mu \cdot \lambda = 1.
$$
While the resonant case is the only case for which the existence
of horizontal disks that avoid the bulging Fatou components of $p$
has been shown, it turns out that it is also the simplest case
where we can show that most fibers $\{w = w_0\}$ contain no such
disks. In fact, in this case the set $E$ of such fibers will not
only have full measure, but will also be open and dense.

\medskip

We use the linearization function $\Phi \colon \mathbb{C}
\rightarrow \mathbb{C}_z$ introduced in Section \ref{section:
Linearization maps}, given in this resonant case by
$$
\Phi(t)=\lim_{n \rightarrow \infty} F^n\left(x_0, \lambda^n
t\right)
$$
By Remark \ref{remark: linearization F} and our assumption that
$F^r(K)$ is transverse to $\Sigma_F^s(0)$, we know that $\Phi$ is
a non-constant holomorphic function.

\begin{prop}\label{prop:resonant}
Let $w_0 \in \mathbb C$ be such that $\Phi(w_0)$ lies in the basin
of infinity of the polynomial $p$. Then the orbit of any
horizontal disk in the $\{w = w_0\}$ fiber must intersect the
bulging Fatou component of $p$.
\end{prop}

Note that under our assumptions the filled Julia set of $p$ has empty interior, hence the basin of infinity is open and dense. Since $\Phi$ is a non-constant holomorphic
function, it follows that the above proposition holds for an open and dense set of parameters $w_0$, and in particular that $F$ does not have wandering Fatou components. The
interested reader should have no difficulty generalizing Proposition \ref{prop:resonant} to the case of several critical points that are all mapped to resonant repelling
periodic points.

\medskip

In preparation for the proof of Proposition \ref{prop:resonant},
we will study the orbits of the critical points $(x_0, \lambda^n
w_0)$ for $w_0$ such that $\Phi(w_0)$ lies in the basin of
infinity of $p$. More precisely, we will show that, for $n \in
\mathbb N$ sufficiently large, the orbits of the critical points
$(x_0, \lambda^n w_0)$ will stay uniformly away from any point in
$\mathcal{J}_p \setminus \{x_0, \ldots, x_r\}$.

\medskip

For $n \in \mathbb{N}$, $m \in \mathbb{Z}$ and $n+m \geq 0$ we define
$$
a_{n,m}=F^{n+m}\left(x_0, \lambda^n w_0 \right).
$$
Note that $a_{n,0}=\Phi_n(w_0)$ and $F(a_{n,m})=a_{n,m+1}$. Since $F(\Phi(t))=\Phi(\tfrac{t}{\lambda})$, we have
$$
\lim_{n \to \infty} a_{n,m}=\Phi\left(\frac{w_0}{ \lambda^m}\right) = F^m\left(\Phi(w_0)\right)
$$
for $m \ge 0$, and
$$
\lim_{n\to \infty} a_{n,-m}=\Phi(\lambda^m w_0) \in
F^{-m}\left(\{\Phi(w_0)\}\right).
$$

\medskip

Denote $a_m=\lim_{n \to \infty} a_{n,m}$. The sequence $(a_m)$ satisfies
$$
\lim_{m \to -\infty} a_m = 0 \textrm{ and } \lim_{m \to +\infty} a_m = \infty.
$$
Recall that $R>0$ was chosen such that $|p(z)| \geq 2 |z|$ for all
$|z|>R$. Increasing $R$ if necessary, we may also assume that
$\mathcal{J}_p\subseteq D(0,R-1)$ and
$$
|F(z,w)|> 2|z|
$$
for all $|z|>R$ and $|w|< |w_0|$. We then define
$$
V=\left\{(z,w)\in\mathbb{C}^2: |z| >R \mathrm{\; and \;} |w|<|w_0|
\right\}.
$$
It follows that $F(V) \subseteq V$. Since $\lim_{m \to +\infty}
a_m = \infty$ there exists an $N_1>0$ such that $|a_m|>R$ for all $m
\geq N_1$. Then for any $\epsilon>0$ we can find a natural number
$N_{\epsilon}$ such that
\begin{enumerate}
\item[(i)]{$|a_{n,m}-a_m|<\epsilon$ whenever $n \geq N_{\epsilon}$, $n+m \geq r$ and $m \leq N_1$.}
\item[(ii)]{$a_{n,m}\in V$ whenever $n \geq N_{\epsilon}$ and $m \geq N_1$.}
\end{enumerate}
Recall that $p^r(x_0) = 0$, which is a repelling fixed point of
$p$. By continuity of $F$ we also have that
\begin{equation}\label{eq: resonant}
\lim_{n \to \infty} a_{n,j-n}=p^j(x_0).
\end{equation}
for all $j \ge 0$. With this information, we can prove the
following lemma.

\begin{lemma}\label{lemma:y}
Let $y \in \mathcal{J}_p \setminus \{x_0,p(x_0), \ldots
p^{r-1}(x_0),0\}$. Then there exist $\delta=\delta(y)>0$ and
$N=N(y) \in \mathbb{N}$ such that $a_{n,m} \notin
\bar{D}(y,\delta) \times \mathbb{C}$ for $n \geq N$ and $n+m \geq
0$.
\end{lemma}

\begin{proof}
Note that all $a_m$ lie in the Fatou set of $p$, and are therefore
not equal to $y$. Since $\lim_{m \rightarrow -\infty} a_m = 0$ and
$\lim_{m \rightarrow \infty} a_m = \infty$ the points $a_m$ must
in fact avoid some small disk $D(y,\delta_1)$. Setting $\epsilon =\delta=
\delta_1/2$, and $N=N_{\epsilon}$, the statement follows whenever $n \geq N$ and $n+m \geq r$. Equation \ref{eq: resonant} allows us to increase $N$ and shrink $\delta$ to deal with the case where $0 \leq n+m <r$.
\end{proof}

\noindent {\bf Proof of Proposition \ref{prop:resonant}.} We
follow the argument due to Fatou that we mentioned in the introduction. Suppose
for the purpose of a contradiction that a horizontal disk $D=
D(x,s) \times \{w_0\}$ converges to the Julia set of $p$. Since this Julia
set has empty interior, we can find a subsequence $(n_j)$ such
that $F^{n_j}(D)$ converges to a point $y \in \mathcal{J}_p$. Since $x_r = 0$
is repelling, we may assume that $y \notin \{x_0, \ldots, x_r\}$.

By Lemma \ref{lemma:y} there exists a cylinder $\bar{D}(y,\delta)
\times \mathbb{C}$ that is avoided by $a_{n,m}$ whenever $n \geq
N$ and $n+m \geq 0$. To use this property, let
$F^N(x,w_0)=:(x',\lambda^N w_0)$ and choose $s'>0$ such that
$D'=D(x',s') \times \{\lambda^N w_0\} \subseteq F^N(D)$. Then
$F^{n_j-N}(D')$ will still converge to $y$.

For $j \in \mathbb{N}$ such that $n_j \geq N$, we consider the
function
$$
\psi_j=F^{n_j-N}(\cdot,\lambda^{N} w_0) \colon \mathbb{C}
\rightarrow \mathbb{C}  \times  \{\lambda^{n_j} w_0\}.
$$
Note that $D' \subseteq \psi_j^{-1}\left(D(y,\delta) \times
\{\lambda^{n_j}w_0\}\right) \subset D(0,R) \times
\{\lambda^{N}w_0\}$ for all sufficiently large values of $j$. Let
$O_j$ be the connected component of $D'$ in this inverse image.
Since the points $a_{n,m}$ avoid $D(y,\delta) \times
\{\lambda^{n_j}w_0\}$, it follows that $\psi_j|_{O_j}$ has no
critical points. It must be a biholomorphism between $O_j$ and
$D(y,\delta) \times \{\lambda^{n_j}w_0\}$, and hence it preserves
the Poincar\'{e} metric. This gives us the inequality
$$
\text{diam}_{D(0,R)}(D') \leq
\text{diam}_{O_j}(D')=\text{diam}_{D(y,\delta)}(F^{n_j-N}(D')).
$$
As $F^{n_j-N}(D')$ converges to $y$, its Poincar\'{e} diameter in
$D(y,\delta)$ must tend to zero. This proves that $D'$ cannot
exist. \hfill $\square$

\section{Semi-hyperbolic polynomials and sublevel sets}\label{section: John Domains}

Let $p$ be a semi-hyperbolic polynomial, with Fatou set
$$
\mathcal{F}_p=I_{\infty} \cup \bigcup_{y} \Omega_p(y)
$$
equal to the union of the basin of infinity and the basins of attracting periodic cycles. Since $p$ is semi-hyperbolic, Siegel disks and parabolic basins cannot occur. We define a set $W_0$ as in Section \ref{section:outline}: Fix $R>0$ large enough
such that
$$
|z| > R \; \; \mathrm{implies \; that} \; \; |p(z)| > 2|z|,
$$
and for each attracting periodic point $y$ of $p$, fix an open neighborhood $W_y$ of its orbit such that $\overline{p(W_y)} \subset W_y$.
Now set
$$
W_0=\{|z|>R\} \cup \bigcup_{y}W_y,
$$
where we take the union over all attracting periodic points $y \in \mathbb{C}$.

Define the sets $U_m:=p^{-m}(W_0)$ and write $V_m$ for their complements. Notice that
$$
\mathcal{F}_p = \bigcup_{m \in \mathbb{N}} U_m.
$$
We will prove the following.

\begin{prop}\label{prop:area}
The areas of the sets $V_m$ decrease exponentially with $m$.
\end{prop}

We note that this proposition certainly does not hold for any polynomial; for example, it does not hold for the polynomial $z \mapsto z+z^2$. We leave the computation to the
reader. To prove the above proposition we will therefore make heavy use of the fact that $p$ is semi-hyperbolic.

\medskip

First of all, the semi-hyperbolicity of $p$ allows us to use Theorem 2.1 in \cite{CJY1994}:
\begin{thm}\label{thm:CJY}
There exist constants $\epsilon>0$, $c>0$ and $\theta<1$ such that for all $x \in \mathcal{J}_p$, we have
$$
\text{diam}\bigl(B_n(x,\epsilon)\bigr) \leq c \theta^n,
$$
where $B_n(x,\epsilon)$ is any connected component of $p^{-n}(D(x,\epsilon))$.
\end{thm}

As $\{z \in \mathbb{C}:~ |z| \leq R \text{ and } \text{d}(z,\mathcal{J}_p) \geq \epsilon \}$ is compact and contained in the increasing union $\bigcup_{m} U_m$, it must be contained in $U_{m_0}$ for some $m_0 \in \mathbb{N}$. Then for $m \geq m_0$, we know that $V_m \subseteq \{z \in \mathbb{C}:~\text{d}(z, \mathcal{J}_p)<\epsilon\}$. Let $z \in V_m$. Then $p^{m-m_0}(z) \in V_{m_0}$ and hence there exists $x \in \mathcal{J}_p$ such that $p^{m-m_0}(z) \in D(x, \epsilon)$. By Theorem \ref{thm:CJY}, this implies that $\text{d}(z, \mathcal{J}_p) \leq c \theta^{m-m_0}$.
It follows that the sets $V_m$ are contained in
$\delta_m$-neighborhoods of $\mathcal{J}_p$, where $\delta_m$ decreases exponentially with $m$.

\medskip

Recall from \cite{CJY1994} that since $p$ is semi-hyperbolic, the basin $I_{\infty}$ is a so-called John domain.
For a John domain in $\mathbb{R}^m$, an upper bound for the Minkowski-dimension of its boundary is given in \cite{KR1997}.
In our circumstances, it follows that $\mathrm{dim}_M(\mathcal{J}_p)=\mathrm{dim}_M(\partial I_{\infty}) < 2$. A quick consequence is the
following.

\begin{lemma}\label{lemma:Mink}
Suppose that the sequence $(\delta_m)$ decreases exponentially fast. Then the area of the $\delta_m$-neighborhoods of $\mathcal{J}_p$ also decreases exponentially fast.
\end{lemma}

\begin{proof}
Let $N_\epsilon$ be the number of $\epsilon$-disks needed to cover $\mathcal{J}_p$. Denote the Minkowski-dimension of $\mathcal{J}_p$ by $h$, and let $h < h^\prime < 2$. Then the
$h^\prime$-dimensional measure of $\mathcal{J}_p$ is $0$, hence for $\epsilon >0$ sufficiently small we have that
$$
N_\epsilon < \frac{1}{\epsilon^{h^\prime}}.
$$
Hence for $m$ sufficiently large the $\delta_m$-neighborhood of $\mathcal{J}_p$ can be covered by $\epsilon^{-h^\prime}=\delta_m^{-h^\prime}$ disks of radius $2 \delta_m$, which means that the area of
this neighborhood is at most
$$
\frac{4\pi \delta_m^2}{\delta_m^{h^\prime}} = 4\pi \delta_m^{2-h^\prime}.
$$
The statement follows.
\end{proof}

\noindent{\bf Proof of Proposition \ref{prop:area}.} Combining Lemma \ref{lemma:Mink} with the fact that the sets $V_m$ are contained in
$\delta_m$-neighborhoods of $\mathcal{J}_p$ with exponentially decreasing $\delta_m$, proves Proposition \ref{prop:area}. \hfill $\square$

\begin{corollary}
We have that
$$
\sum_{m \in \mathbb N} \mathrm{Area}(V_m) < \infty.
$$
\end{corollary}

It is not reasonable to expect this corollary to hold for arbitrary polynomials, as $\mathcal{J}_p$ may have positive measure. It would be interesting to know whether $\sum_{m \in \mathbb N} \mathrm{Area}(V_m)$ is necessarily finite when $\mathcal{J}_p$ is assumed to have Hausdorff dimension strictly less than $2$.
Note that if the filled Julia set $K_p$ has no interior, then $V_m$ can be defined by $\{G \le \frac{1}{d^m}\}$, where $G$ is the Green's function. Leaving the setting of Green's functions arising from complex dynamics, we note that there do exist compact subsets $K \subset \mathbb C$ with Hausdorff dimension strictly smaller than $2$ for which the Green's function $G_K$ with logarithmic pole at infinity satisfies
$$
\sum_{m \in \mathbb N} \mathrm{Area} \{G_K \leq \frac{1}{2^m} \}=
\infty.
$$
Such a set $K$ can for example be constructed by taking the boundaries of a nested sequence of disks, and removing from these circles progressively smaller intervals. We thank S{\l}awomir Ko{\l}odziej for pointing out this construction to us.

\section{Forcing escape of critical points}\label{section:tracking}

In this section we will prove the estimate in Proposition
\ref{prop:final critical pointcount} as mentioned in the outline.
We consider the family of forward iterates $F^n$ restricted to a
fiber $\{w = w_0\}$ for a fixed $w_0 \neq 0$. Writing $w_n =
\lambda^n w_0$ and $f_n(z) = f_{w_n}(z) = f(z,w_n)$ we can
consider these iterates $F^n$ as compositions of a sequence of
polynomials, i.e.,
$$
F^n|_{\{w = w_0\}} = f_{n-1} \circ \cdots \circ f_0.
$$
Every polynomial $f_{w_i}$ introduces new critical points. It
turns out that by choosing $w_0$ carefully we can make sure that
all these critical points escape to the set $W$ defined in Section \ref{section:outline}, and moreover obtain
some control on the rate at which the critical points escape. In
fact, we will prove the existence of a subset $E \subset \mathbb
C_w$ of full measure for which \emph{all critical points} of maps
$f_{w_n}$, for $w \in E$, escape to $W$ at least at some
moderate rate, while \emph{most critical points} escape very fast.

\medskip

The critical points of $F$ form a finite union of irreducible
varieties in $\mathbb{C}^2$, as described in Section \ref{section:
Linearization maps}. We will study these varieties one at a time.
Given a fixed open neighborhood $U\subseteq \mathbb{C}$ of the
postcritical set of the polynomial $p(\cdot) = f(\cdot,0)$, we
will prove the following:

\begin{prop}\label{prop:near-final critical pointcount}
Let $K$ be an irreducible critical variety. Then there exists a set $E_K\subseteq \mathbb{C}$, of full measure in some neighbourhood of the origin, with the following property:

For all $\eta>0$ and $w \in E_K$ we have:
$$
\#\left\{s \leq n:~ \exists (z,\lambda^s w)\in K \mathrm{\, with
\,} F_1^{n-s}(z, \lambda^s w) \notin W_0 \cup U \right\} \leq
\eta \log n +C_{4},
$$
where $C_{4}$ is a constant depending on $K$, $\eta$, $U$ and $w$.
\end{prop}

A consequence is the main result of this section.

\begin{prop}\label{prop:final critical pointcount}
There exists a set $E\subset \mathbb{C}$, of full measure in some neighborhood of
the origin, with the following property: For all $w \in E$ there
exists a constant $C_2(w,U)$ such that for all $n \in \mathbb{N}$
we have
$$
\#\left\{z:~ \frac{\partial F_1^n}{\partial z}(z,w)=0 \mathrm{\;
and \;} F_1^n(z,w) \notin W_0 \cap U \right\} \leq
C_2\sqrt{n}.
$$
\end{prop}
\begin{proof}
Define $E=\bigcap_{K}E_K$, let $w \in E$ and let $\eta>0$. Pick
$\epsilon'
>0$ such that $E$ has full measure in $D(0,\epsilon')$. We have at
most finitely many critical varieties: $K_1$ up to $K_q$. Applying
Proposition \ref{prop:near-final critical pointcount} for each of
them gives us:
$$
\#\left\{s \leq n:~ \exists (z,\lambda^s w) \in \bigcup K_i
\mathrm{\,with\,} F_1^{n-s}(z, \lambda^s w) \notin W_0 \cup U
\right\}  \leq q\eta \log n +C,
$$
where $C$ depends only on $\eta$, $U$ and $w$. As
$f_w(z)=z^d+a_{d-1}(w)z^{d-1}+\ldots+a_0(w)$, each $\{\lambda^s
w\}$-fiber contains $d-1$ critical points of $f_w(z)$, counting with
multiplicities. Then the restricted function
$$
F_1^n \colon (F^n)^{-1}\bigl(\left(\mathbb{C} \setminus (W_0 \cup U)\right) \times
\{\lambda^n w\} \bigr) \rightarrow \left(\mathbb{C} \setminus (W_0 \cup U)\right)
$$
is a composition of $n$ holomorphic functions. At most $q\eta \log
n +C$ of them can each have at most $d-1$ critical points. The total
number of critical points of this restricted function is therefore
at most
$$
d^{q\eta \log n +C}=d^{C}n^{q\eta \log d}.
$$
The desired result is obtained by choosing $\eta$ small enough.
\end{proof}

The rest of this section is devoted to proving Proposition
\ref{prop:near-final critical pointcount}. This proposition will
be proved in two parts. First, we will prove that for almost every
$w_0 \in \mathbb C$ and all $n\in \mathbb N$, all critical points
of $f_{w_n}$ escape to the set $W$ in at most $C \cdot \log(n)$
steps, for some uniform $C>0$. We will not be able to control the
constant $C$ though, hence we will not yet be able to deduce that
the estimate in Proposition \ref{prop:near-final critical
pointcount} holds for \emph{any} $\eta >0$. However, we will also
prove that for almost every $w_0 \in \mathbb C$ and \textsl{most}
$n \in \mathbb N$, the critical points of $f_{w_n}$ escape to the
set $W$ in a number of steps that does not depend on $n$. By
choosing a sufficiently strong interpretation of ``most $n \in
\mathbb N$'' it will follow that $\eta$ in Proposition
\ref{prop:near-final critical pointcount} can be chosen
arbitrarily small.

\medskip

Let $K$ be one of the irreducible analytic varieties of critical
points of $F$ in $\mathbb{C} \times D(0,\epsilon)$, as
described in Remark \ref{remark: linearization F}. Then $K$
intersects the $\{w=0\}$-fiber in a single point $x_0 \in \mathbb{C}_z$. If $x_0 \in \mathcal{F}_p$, then $F^r(x_0) \in W_0$ for some $r \in \mathbb{N}$. Shrinking $\epsilon$ if necessary, we find that $F^r(K) \subseteq W$. In this setting, Proposition \ref{prop:near-final critical
pointcount} follows immediately. We will therefore assume that $x_0 \in
\mathcal{J}_p \times \{0\}$ from now on. Its image $x_r=p^r(x_0)$
is a repelling fixed point. We define $\mu=p'(x_r)$.

Since $K$ intersects $\{w=0\}$ in a unique point, we can locally parameterize $K$ by
$$
K = \left\{\bigl( \gamma_1(u),\gamma_2(u)\bigr): u \in
D\bigl(0,\epsilon^{1/l}\bigr)\right\},
$$
where $\gamma_2(t)=t^l$ for some $l \geq 1$.

\medskip

If $F^r(K)$ is contained in the stable variety $\Sigma_F^s(x_r)$
of $F$ at $x_r$, then so is $F^n(K)$ for any $n \geq r$. The
points in $K$ all converge to $(x_r,0)$ along this stable variety.
In this setting, Proposition \ref{prop:near-final critical
pointcount} follows immediately.

For the rest of this section, we will therefore assume that
$F^r(K) \nsubseteq \Sigma_F^s(x_r)$. Remark \ref{remark:
linearization F} then gives us the linearization maps
$$
\Phi_j(t)=F^{kj}\left(\gamma_1(\mu^{-j}t),(\mu^{-j}t)^l
\right)
$$
that converge uniformly and exponentially fast on compact subsets,
to a holomorphic function $\Phi \colon \mathbb{C} \rightarrow
\mathbb{C}_z$ of local degree $k$ for some $k \geq 1$.

\medskip

We wish to study the points of the critical variety $K$ in a
sequence of fibers $\mathbb{C} \times \{w\}$, $\mathbb{C} \times
\{\lambda w\}$, $\mathbb{C} \times \{\lambda^2 w\}$, $\ldots$
where $w \in D(0,\epsilon)$. This will be easier if we look at
pre-images of these critical points under $\gamma$: pick $\nu \in
\mathbb{C}$ such that $\nu^l=\lambda$ and consider the sequences
$u, \nu u, \nu^2 u \ldots$ for $u \in D(0,\epsilon^{1/l})$. By
abuse of notation, we replace the constant $\epsilon^{1/l}$ by the
letter $\epsilon$ again. For $u \in D(0,\epsilon)$, the orbit of
$\gamma(\nu^s u)$ under $F$ can now be studied using the functions
$\Phi_j$ and $\Phi$:
\begin{align*}
F^n\bigl(\gamma(\nu^s u)\bigr)&=F^{n-kj}\left(\Phi_j(\mu^j \nu^s
u)\right) \text{ whenever } n \geq kj\\
&\approx F^{n-kj}\left(\Phi(\mu^j \nu^s u)\right).
\end{align*}
In order to use exponential estimates on the rate of convergence
of $\Phi_j$ to $\Phi$, we ask that $|\mu^j \nu^s u| \leq \epsilon$,
while having $j$ run to infinity. We therefore define $j(s)$ to be
the unique integer for which
$$
|\mu|^{-1}<|\mu^{j(s)} \nu^s| \leq 1.
$$
Lemma \ref{lemma:Cirkelverlaten} now implies the following:

\begin{corollary}\label{cor:j(s) en m}
There exist constants $C_5, C_6 > 0$ such that for all $s,m \in
\mathbb{N}$ satisfying $m < C_5j(s)-C_6$ and for all $u \in
D(0,\epsilon)$, we have:
$$
\text{If }\Phi\bigl(\mu^{j(s)} \nu^s u\bigr) \in U_m, \quad \text{then }
F^{m+1+kj(s)}\bigl(\gamma(\nu^s u)\bigr) \in W.
$$
\end{corollary}

\begin{proof}
If $\Phi(\mu^{j(s)} \nu^s u) \in U_m$ and $j(s)$ is large enough
to ensure that $||\Phi_{j(s)}-\Phi||_{D(0,\epsilon)} \leq C_1
M^{-m}$, then $F^{m+1+kj(s)}\left(\gamma(\nu^s
u)\right)=F^{m+1}\left(\Phi_{j(s)}(\mu^{j(s)}\nu^s u) \right) \in
W$ by Lemma \ref{lemma:Cirkelverlaten}. By the uniform convergence of $\Phi_j$ to $\Phi$ on
$D(0,\epsilon)$, the requirement that
$||\Phi_{j(s)}-\Phi||_{D(0,\epsilon)} \leq C_1 M^{-m}$ can be
translated to $m<C_5j(s)-C_6$ for some uniform constants $C_5$ and
$C_6$.
\end{proof}

Corollary \ref{cor:j(s) en m} motivates us to take a closer look
at $\left(\Phi|_{D(0,\epsilon)}\right)^{-1}(U_m)$.

\medskip

\subsection{All critical points escape slowly}

The complement of $\left(\Phi|_{D(0,\epsilon)}\right)^{-1}(U_m)$
within $D(0,\epsilon)$ is exactly
$\left(\Phi|_{D(0,\epsilon)}\right)^{-1}(V_m)$. By Proposition
\ref{prop:area} the area of $V_m$ decreases exponentially with
$m$. As $\Phi$ is a nonconstant holomorphic function, the area of
$\left(\Phi|_{D(0,\epsilon)}\right)^{-1}(V_m)$ also decreases
exponentially with $m$, although not necessarily at the same rate.

\medskip

Using this information, we can apply Corollary \ref{cor:j(s) en m}
to show that given $s \in \mathbb{N}$, the critical point
$\gamma(\nu^s u)$ moves away from the Julia set $\mathcal{J}_p$
fairly quickly for most values $u \in D(0,\epsilon)$.

\begin{lemma}\label{lemma:goede t1}
There exist uniform constants $C_7>0$ and $\alpha<1$ such that
$$
\textrm{Area}\left\{u \in D(0,\epsilon) ~:~
 F^{kj(s)+m+1}\bigl(\gamma(\nu^s u)\bigr) \notin W \right\} \leq C_7 \alpha^m
$$
for all $s,m \in \mathbb{N}$ with $m<C_5j(s)-C_6$.
\end{lemma}
\begin{proof}
Corollary \ref{cor:j(s) en m} implies the inclusion
\begin{align*}
\left\{u \in D(0,\epsilon) ~:~ F^{N+kj(s)+m}\bigl(\gamma(\nu^s
u)\bigr) \notin W \right\} &\subseteq \left\{u \in D(0,\epsilon) ~:~
\Phi\bigl(\mu^{j(s)} \nu^s u\bigr) \notin U_m\right\}\\
&= D(0,\epsilon) \cap \mu^{-j(s)}\nu^{-s} \Phi^{-1}(V_m)\\
 &\subseteq
\mu^{-j(s)}\nu^{-s} \left(\Phi|_{D(0,\epsilon)}\right)^{-1}(V_m).
\end{align*}
The area of this last set is at most $|\mu|^{-1}
\textrm{Area}\left(\left(\Phi|_{D(0,\epsilon)}\right)^{-1}(V_m)\right)$,
which decreases exponentially with $m$.
\end{proof}

We will use Lemma \ref{lemma:goede t1} to find values of $u\in
D(0,\epsilon)$ for which nearly all critical points
$\{\gamma(\nu^s u)\}_{s \in \mathbb{N}}$ move away from
$\mathcal{J}_p$ fairly fast. Our strategy is to apply the
following well-known lemma:

\begin{lemma}\label{lemma:areasum}
Let $\{A_s\}_{n \in \mathbb{N}}$ be a sequence of subsets of a
measure space $X$ such that
$$
\sum_{s \in N} \textrm{measure}(A_s) < \infty.
$$
Then the set $E=\{u \in X ~|~ \exists N_u \in \mathbb{N},~ \forall
s \geq N_u,~ u \notin A_s \}$ has full measure.
\end{lemma}

To be able to use Lemma \ref{lemma:goede t1} in the setting of
Lemma \ref{lemma:areasum}, we need to pick an appropriate value
$m(s)$ for each $s \in \mathbb{N}$. Then we can apply Lemma
\ref{lemma:areasum} to the sets
$$
A_s=\left\{u \in D(0,\epsilon) ~:~
 F^{kj(s)+m(s)+1}\bigl(\gamma(\nu^s u)\bigr) \notin W \right\}.
$$
To make sure that the areas of the sets $A_s$ have a finite sum,
we need to pick $m(s)$ such that
$$
\sum_{s \in \mathbb{N}} C_7 \alpha^{m(s)} < \infty.
$$
On the other hand, it would be preferable to pick $m(s)$ as low as
possible. This gives stronger information on the rate at which a
critical point $\gamma(\nu^s u)$ (for $u \notin A_s$) moves
towards $W$. And finally, $m(s)$ needs to satisfy
$m(s)<C_5j(s)-C_6$ to be allowed to use the estimate in Lemma
\ref{lemma:goede t1}.

To satisfy all these requirements we define
$$
m(s) := \left\lceil \frac{2 \log s}{\log \alpha} -C_8
\right\rceil,
$$
where the constant $C_8>0$ is chosen such that $m(s)<C_5j(s)-C_6$
holds for all $s \in \mathbb{N}$. As $\alpha^{-m(s)}\leq
\alpha^{C_8} s^{-2}$, the areas of the sets $A_s$ will have their
finite sum. We have now proved the following corollary:

\begin{corollary}\label{cor:E full measure}
The set
$$
E_1(K)=\left\{u \in D(0,\epsilon)~\Big|~ \exists N_u \in \mathbb{N},~
\forall s \geq N_u,~ F^{kj(s)+m(s)+1}\bigl(\gamma(\nu^s u)\bigr) \in W
\right\}
$$
has full measure in $D(0,\epsilon)$.
\end{corollary}

For points $\nu^s u$ with $u \in E_1(K)$ and $s \geq N_u$, we now
have a reasonable grasp of the behavior of $F^n(\gamma(\nu^s u))$
when $n \geq kj(s)$: the point will move to $W$ in at most another
$N+m(s)$ applications of the function $F$. The next lemma will
take a look at the setting when $n<kj(s)$. For any neighborhood
$U\subseteq \mathbb{C}$ of the post-critical set of $p$, we have the following:

\begin{lemma}\label{lemma:kleine n}
For $u \in D(0,\epsilon)$ we have
$$
F_1^n\bigl(\gamma(\nu^s u)\bigr) \in U
$$
whenever $C_9 \leq n < kj(s)-C_9$, for a constant $C_9(\epsilon,
U)$ that is independent of $s$.
\end{lemma}
\begin{proof}
Pick $\rho >0$ such that $D(x_r,\rho) \subseteq U \cap D(0,R)$.
Fix $\theta$ such that $\Phi(D(0,\theta)) \subseteq D(x_r,
M^{-k}\rho/2)$. Write $n=ki+v$ with $0 \leq v <k$, and note that
$$
F^n\bigl(\gamma(\nu^s u)\bigr)=F^v\bigl(\Phi_i(\mu^i \nu^s u)\bigr).
$$
Assume that $i$ is small enough to have $|\mu^i \nu^s
\epsilon|<\theta$. On the other hand, assume that $i$ is large
enough to ensure that $||\Phi_i - \Phi||_{D(0,\epsilon)} <
M^{-k}\rho/2$. Then for $u \in D(0,\epsilon)$ we have
$\textrm{d}\left(\Phi_i(\mu^i \nu^s u),x_r\right) < M^{-k}\rho$.
As $x_r$ is a fixed point of $F$ and $v<k$, this implies that
$\textrm{d}\left(F^v(\Phi_i(\mu^i \nu^s u)),x_r\right) < \rho$ and
hence $F^n(\gamma(\nu^s u)) \in U$. Here, we used the fact that
$D(x_r,\rho) \subseteq D(0,R)$ and that $M$ is an upper bound for
$|DF|$ on $D(0,R) \times D(0,\epsilon)$.

As $|\mu^{j(s)}\nu^s|<1$, the requirements on $i$ can be
translated to $C_9 \leq n < kj(s)-C_9$ for some constant
$C_9(\epsilon,U)$.
\end{proof}

Combining Corollary \ref{cor:E full measure} with Lemma
\ref{lemma:kleine n} allows us to prove the slow escape of the points in the critical variety $K$:

\begin{prop}\label{prop:log n}
For any $u \in E_1(K)$ and $n \in \mathbb{N}$, we have
$$
\#\left\{s \leq n:~ F_1^{n-s}\bigl(\gamma(\nu^s u)\bigr) \notin W_0 \cup U \right\} \leq C_{10} \log n +C_{11},
$$
for constants $C_{10}(u,K,U,\epsilon)$ and
$C_{11}(u,K,U,\epsilon)$ independent of $n$.
\end{prop}

\begin{proof}
Write $S_n=\left\{s \leq n:~ F_1^{n-s}(\gamma(\nu^s u)) \notin W_0 \cup U \right\}$. By definition of $E_1(K)$ and the fact that
$F(W) \subseteq W$, we know that $s \notin S_n$ when $s$ satisfies
both $s \geq N_u$ and $n-s \geq kj(s)+m(s)+1$.

By Lemma \ref{lemma:kleine n} we know that $s \notin S_n$ when
$C_9 \leq n-s < kj(s)-C_9$. This leaves us with $2C_9+N_u$
possible entries in $S_n$, plus all values $s$ for which
$$
kj(s) \leq n-s <kj(s)+m(s)+1.
$$
Such a value of $s$ must also satisfy
\begin{align*}
n &\geq s+kj(s) &>n-m(n)-1,&\text{ and hence}\\
n &\geq s+k\left\lfloor \frac{s}{l}
\frac{\log|\lambda^{-1}|}{\log|\mu|} \right\rfloor
&>n-m(n)-1,&\text{
and }\\
n+k &\geq s\left(1+\frac{k}{l}
\frac{\log|\lambda^{-1}|}{\log|\mu|}\right) &>n-m(n)-1.&
\end{align*}
Plugging in our formula for $m(n)$ shows that this gives us a
sequence of at most
$$
\frac{k+1+\left\lceil \frac{2 \log n}{\log \alpha} -C_8
\right\rceil}{1+\frac{k}{l}
\frac{\log|\lambda^{-1}|}{\log|\mu|}}$$ consecutive values of $s$.
This proves the proposition.
\end{proof}

\subsection{Most critical points escape quickly}

Proposition \ref{prop:log n} and its proof showed that for each $u
\in E_1(K)$ and $n \in \mathbb{N}$, the set
$$
S_n(u)=\left\{s \leq n:~ F_1^{n-s}\bigl(\gamma(\nu^s u)\bigr) \notin W_0 \cup U \right\}
$$
could contain at most $C_{11}$ points, plus possibly some integers
in an interval of length $C_{10} \log n$. Very little information
about the dynamical system was used to prove this. In this subsection, we will show
that for almost all $u \in E_1(K)$, an arbitrarily large portion
of the values $s$ in any sufficiently large interval do not lie in
$S_n$.

\medskip

Recall that $j(s)$ is the unique natural number
such that
$$
 |\mu^{-1}|<|\mu^{j(s)} \nu^s |\leq 1.
$$
The sequence $\{\mu^{j(s)} \nu^s\}_{s \in \mathbb{N}}$ lives in
the annulus $\{|\mu^{-1}|<|t| \leq 1\}$. Multiplication by $\mu$
allows us to identify the inner and outer boundaries of the
annulus, giving us the torus $\left(\mathbb{C} \setminus
\{0\}\right)/\mu^{\mathbb{Z}}$. It will be more convenient to
think about this torus additively. Write
$$
T=\mathbb{C} / (2\pi i \mathbb{Z}+\log \mu \mathbb{Z}).
$$
It does not matter which value for $\log \mu$ we use, since $2 \pi
i$ is included in our lattice anyway. As $|\mu|>1$, we know that
$\log \mu$ is not purely imaginary. Composing a projection with
any branche of the logarithm, we obtain a holomorphic covering
map:
$$
\Psi \colon \quad \mathbb{C} \setminus \{0\}
\stackrel{\log}{\longrightarrow} \mathbb{C}
\stackrel{\pi}{\longrightarrow} \mathbb{C} / (2\pi i
\mathbb{Z}+\log \mu \mathbb{Z})=T.
$$
The projection $\pi$ ensures that the branche of the logarithm is
irrelevant. When restricted to $\{|\mu^{-1}|<|t| \leq 1\}$, the
map $\Psi$ is a bijection.

\medskip

Let $y$ be any value for $\log \nu$. Then in the torus $T$, the
sequence $\{\Psi(\mu^{j(s)} \nu^s)\}_{s \in \mathbb{N}}$ is
represented by $\{s y\}_{s \in \mathbb{N}}$. We recall the following lemma from Section 1.4 of \cite{KH1995}.

\begin{lemma}
Let $T$ be an additive torus and $y \in T$. Then the set $\overline{\{s y\}_{s \in \mathbb{N}}}$ is a subgroup of $T$, and one of the following statements must be true:
\begin{enumerate}
\item[(1)]{The set $\overline{\{s y\}_{s \in \mathbb{N}}}$ is finite.}
\item[(2)]{The set $\overline{\{s y\}_{s \in \mathbb{N}}}$ is
one-dimensional. It is a disjoint union of finitely many evenly
spaced lines/circles in $T$.}
\item[(3)]{The set $\overline{\{s y\}_{s \in \mathbb{N}}}$ equals $T$.}
\end{enumerate}
\end{lemma}

The torus $T$ can be covered by cosets
$[a]=a+\overline{\{s y\}_{s \in \mathbb{N}}}$ for $a \in T$.
Denote by $\sigma_a$ the counting measure, the
one-dimensional Lebesgue measure, or the two-dimensional Lebesgue
measure on $[a]$ for cases  (1), (2) or (3) respectively. The action of $x \mapsto x+y$ on $[a]$ is uniquely ergodic, hence we have:

\begin{lemma}\label{lemma:distributie}
Let $\eta >0$, $a \in T$, and $O \subseteq [a]$ open. Then there
exists $C_{12}(\eta,O,T)>0$ such that for all $r \in \mathbb{N}$
and $n \geq C_{12}$ we have
$$
\frac{1}{n} \# \left\{r < s \leq r+n: a+sy \in O \right\} > \frac{\sigma_a(O)}{\sigma_a([a])}-\eta.
$$
\end{lemma}

See for example Section 4.3 in \cite{EW2011}.

\medskip

Since $\mathcal{J}_p$ is $F$-invariant and $\Phi$ satisfies the
equation $F^k \circ \Phi (t) = \Phi(\mu t)$, the set
$\Phi^{-1}(\mathcal{J}_p)$ must be invariant under multiplication
by $\mu$. By our assumptions $\mathcal{J}_p$ has measure zero, so
since $\Phi$ is a non-constant holomorphic function, the set
$\Phi^{-1}(\mathcal{J}_p)\cap D(0,1)$ must also have measure zero.
Note that $\Phi^{-1}(\mathcal{J}_p)$ is also closed. It follows
that $\Psi \left(\Phi^{-1}(\mathcal{J}_p)\right)$ is a compact set
of measure zero.

\medskip

For any $a \in T$, the set $\Psi
\left(\Phi^{-1}(\mathcal{J}_p)\right)\cap [a]$ will be a closed
subset. By Fubini's Theorem, it follows that $\sigma_a\left(\Psi
\left(\Phi^{-1}(\mathcal{J}_p)\right)\cap [a]\right)=0$ for almost
every $a \in T$. Denote the set of these values of $a$ by $A$.

Let $\eta >0$. For each $a \in A$, we can choose an open set
$O(a,\eta)\subseteq [a]$ such that $\sigma_a(O) \geq (1-\eta)\sigma_a([a])$ and
$\overline{O} \cap \Psi \left(\Phi^{-1}(\mathcal{J}_p)\right)=\emptyset$. By Lemma \ref{lemma:distributie} we can find
$C_{12}(a,\eta)$ such that
$$
\frac{1}{n}\# \left\{r<s \leq r+n: a+sy \in O \right\}> 1-2\eta
\quad \text { for all }r\in\mathbb{N},n\geq C_{12}.
$$
Define $E_2(K)=\{u \in D(0,\epsilon):u\neq 0 \text{ and } \Psi(u)
\in A\}$. As $A$ has full measure in $T$, the set $E_2(K)$ has
full measure in $D(0,\epsilon)$. For any $u \in E_2(K)$, $r \in
\mathbb{N}$, and $n\geq C_{12}$ we can translate the above
statement to:
$$
\frac{1}{n}\# \left\{r<s \leq r+n: \mu^{j(s)}\lambda^s u \in
\Psi^{-1}(\overline{O})\cap\{|\mu^{-1}u| \leq |t| \leq |u|\}
\right\}> 1-2\eta.
$$
The set $\Psi^{-1}(\overline{O})\cap\{|\mu^{-1}u| \leq |t| \leq
|u|\}$ is closed and bounded, hence compact. Its intersection with
$\Phi^{-1}(\mathcal{J}_p)$ is empty. We can therefore cover
$\Psi^{-1}(\overline{O})\cap\{|\mu^{-1}u| \leq |t| \leq |u|\}$ by
the increasing open sets $\{\Phi^{-1}(U_m)\}_{m \in \mathbb{N}}$.
By compactness, we can find $\ell=\ell(u,\eta)$ such that
$\Psi^{-1}(\overline{O})\cap\{|\mu^{-1}u| \leq |t| \leq
|u|\}\subseteq \Phi^{-1}(U_\ell)$.

\medskip

We can now prove the following proposition:

\begin{prop}\label{prop:dynamics upper bound}
For all $\eta>0$ and $u \in E_2(K)$, we have:
$$
\frac{1}{n}\# \left\{r<s \leq r+n: F^{kj(s)+\ell+1}\bigl(\gamma(\lambda^s u)\bigr) \in W \right\}>
1-2\eta,
$$
whenever $n \geq C_{12}$ and  $r \geq C_{13}$. Here
$\ell=\ell(u,\eta)$ and $C_{12}(u,\eta)$ are as defined above and
$C_{13}$ depends on $\ell$.
\end{prop}

\begin{proof}
Let $\eta>0$, $u \in E_2(K)$ and $n \geq C_{12}(u,\eta)$. We start
by the observation that
\begin{align*}
1-2\eta &< \frac{1}{n}\# \left\{r<s \leq r+n: \mu^{j(s)}\lambda^s u \in \Phi^{-1}(U_\ell) \right\}\\
&=\frac{1}{n}\# \left\{r<s \leq r+n: \Phi(\mu^{j(s)}\lambda^s u) \in U_\ell \right\}
\end{align*}
for all $r \in \mathbb{N}$. When $r$ is sufficiently large, $j(s)$
will be large enough to ensure that the latter set is included in
$$
\frac{1}{n}\# \left\{r<s \leq r+n:
d\left(\Phi_{j(s)}(\mu^{j(s)}\lambda^s u), U_\ell
\right)<C_1M^{-\ell} \right\},
$$
and by Lemma \ref{lemma:Cirkelverlaten}, this set is contained in the set
$$
\frac{1}{n}\# \left\{r<s \leq r+n: F^{\ell+1}\left(\Phi_{j(s)}(\mu^{j(s)}\lambda^s u)\right) \in W \right\}>
1-2\eta.
$$
Combining all this with the definition of $\Phi_{j(s)}$ proves the proposition.
\end{proof}

Define $E_3(K)=E_1(K) \cap E_2(K)$. This set still has full measure in $D(0,\epsilon)$, and we can combine Propositions \ref{prop:log n} and \ref{prop:dynamics upper bound} to
find:

\begin{corollary}\label{cor:E_3(K)}
Let $\eta>0$ and $u \in E_3(K)$. Then there exists a constant
$C_{14}$, depending on $\eta$, $u$, $U$, $\epsilon$ and $K$, such that for all $n
\in \mathbb{N}$ we have
$$
\#\left\{s \leq n:~ F_1^{n-s}\bigl(\gamma(\nu^s u)\bigr) \notin W_0 \cup U \right\} \leq \eta \log n +C_{14}.
$$
\end{corollary}

\begin{proof}
As in the proof of Proposition \ref{prop:log n},
write
$$
S_n=\left\{s \leq n:~ F_1^{n-s}\bigl(\gamma(\nu^s u)\bigr) \notin W_0
\cup U \right\}.
$$
We showed that $S_n$ contains $2C_9+N_u$ entries, plus possibly the values $s$ for which
$$
kj(s) \leq n-s <kj(s)+m(n)+1.
$$
The length of this interval was bounded by $C_{10} \log n$, for a constant $C_{10}$ depending on $u$, $K$, $U$ and $\epsilon$.
Choose $\theta >0$ such that $2 \theta C_{10} < \eta$ and apply Proposition \ref{prop:dynamics upper bound}:

Whenever $n-s \geq kj(s)+\ell+1$ and $s \geq C_{13}$, no interval of length $L \geq C_{12}$ will contain more than $2\theta L$ elements of $S_n$.
We can now pick $C_{14}$
to be large enough to ensure that $\#S_n \leq \eta \log n +C_{14}$
for all $n \in \mathbb{N}$.
\end{proof}

\noindent {\bf Proof of Proposition \ref{prop:near-final critical
pointcount}.} Write $Q=D(0,\epsilon) \setminus E_3(K)$, a set of
measure zero. The set $Q^l=\{u^l : u \in Q\}$ still has measure
zero. Define $E_K=D(0,\epsilon^l) \setminus Q^l$. Then for each
$w \in E_K$ we have
\begin{align*}
\#\left\{s \leq n:~ \exists (z,\lambda^s w)\in K \text{ with
}F_1^{n-s}(z, \lambda^s w) \notin W_0 \cup U \right\} &= \\
\bigcup_{u: \; u^l=w} \#\left\{s \leq n:~ F_1^{n-s}\bigl(\gamma(\nu^s
u)\bigr) \notin W_0 \cup U \right\} & \leq l\eta \log n +C_4.
\end{align*}
for a constant $C_4=C_4(K,\eta,U,w)$. \hfill $\square$

\section{Hyperbolic areas of inverse images}\label{section:areas}

Only Proposition \ref{prop:area inv} remains to be shown. First
recall the following similar result from \cite{SL}.

\begin{lemma}\label{lemma:SL}
For any $0<r<1$ and any holomorphic proper map $f:V \rightarrow
\mathbb D$ of degree $d$, with $V$ simply connected, each
connected component of $f^{-1}(\overline{D_r}(0))$ has diameter at
most
$$
2d\log\left(\frac{1+r^{1/d}}{1-r^{1/d}}\right)
$$
with respect to the Poincar\'e metric on $V$.
\end{lemma}

Our goal is to prove a similar estimate regarding hyperbolic areas
rather than diameters. Our strategy will be to remove small
neighborhoods of critical points, and prove that the area of the
inverse image outside of these neighborhoods must be small. We
first give an estimate on the proximity of a critical point for
points with small derivatives.

\begin{lemma}\label{lemma:distance}
If $f(z) \colon \mathbb{D} \rightarrow \mathbb{D}$ is a proper
holomorphic map of degree $d$ satisfying $f(0)=0$ and $|f'(0)|\leq
\delta^{d-1}$ with $\delta<1/4$, then $f$ has a critical point $c$
with $d_{\mathbb{D}}(c,0)\leq 32d\delta$.
\end{lemma}
\begin{proof}
Without loss of generality may write
$$
f(z)= z\prod_{i=1}^{d-1}\frac{z-a_i}{1-\bar{a}_iz}.
$$
We may assume that $|a_1| \leq |a_j|$ for $j \ge 2$. By our
assumption that $\delta^{d-1} \geq
|f'(0)|=|\prod_{i=1}^{d-1}a_i|$, it follows that $|a_1|<\delta$.
For $z=ta_1$ with $t \in [0,1]$ we have
\begin{align*}
f(z)&=z\prod_{i=1}^{d-1}\frac{z-a_i}{1-\bar{a}_iz}=z(z-a_1)\prod_{i=2}^{d-1}(z-a_i) \prod_{i=1}^{d-1}\frac{1}{1-\bar{a}_iz}, \; \textrm{ and hence}\\
|f(z)| &\leq
\frac{|a_1|^2}{4}\prod_{i=2}^{d-1}(2|a_i|)\prod_{i=1}^{d-1}\frac{1}{1-\delta}
 \leq \left(\frac{2\delta}{1-\delta}\right)^d.
\end{align*}

 Let $r=\left(\frac{2\delta}{1-\delta}\right)^d$. By the above estimate the set $f^{-1}\left(D(0,r)\right)$ contains the interval between $0$ and $a_1$. As both endpoints of
 this interval are mapped to $0$ and $D(0,r)$ is simply connected, the connected component of $f^{-1}\left(D(0,r)\right)$ that contains this interval must also contain a
 critical point $c$. By Lemma \ref{lemma:SL} the Poincar\'{e} distance from $c$ to $0$ can therefore be estimated by
$$
2d\log\left(\frac{1+r^{1/d}}{1-r^{1/d}}\right)=2d\log\left(\frac{1+\frac{2\delta}{1-\delta}}{1-\frac{2\delta}{1-\delta}}\right)\leq2d \cdot 4 \frac{2\delta}{1-\delta}\leq
32d\delta
$$
For the first inequality, we used our assumption that $\delta <
1/4$.
\end{proof}

Let $f\colon \mathbb D \rightarrow \mathbb D$ be a proper holomorphic
map of degree $d$. Suppose that $x \in \mathbb{D}$ satisfies
$d_{\mathbb{D}}(x,c_i) > 32d\delta$ for each critical point $c_i$.
Write $f(x)=y$ and define the functions
$g(z)=\frac{z+x}{1+\bar{x}z}$ and $h(z)=\frac{z-y}{1-\bar{y}z}$.
Since $g$ and $h$ preserve the Poincar\'{e} metric, the point
$0=g^{-1}(x)$ has Poincar\'e distance at least $32d\delta$ to any
critical point of $h\circ f \circ g$, and Lemma
\ref{lemma:distance} gives
$$
\delta^{d-1} \leq |(h\circ f \circ
g)'(0)|=|h'(y)||f'(x)||g'(0)|=\frac{1-|x|^2}{1-|y|^2}|f'(x)|.
$$

\begin{prop}\label{prop:area inv}
Let $f\colon \mathbb D \rightarrow \mathbb D$ be a proper holomorphic
map of degree $d$, and let $R \subseteq \mathbb{D}$ have
Poincar\'{e} area $A$. If $d\cdot A^{1/2d}<1/8$, the inverse image
of $R$ under $f$ will have Poincar\'{e} area at most $C_3d^3
A^{1/d}$ for a uniform constant $C_3$.
\end{prop}

\begin{proof}
Define $\delta=A^{1/2d}$ and let $S=\{z\in \mathbb{D}~|~
d_{\mathbb{D}}(z,c_i)>32d\delta \textrm{ for all } i \}$. Then
$f^{-1}(R) \subseteq f|_S^{-1}(R) \cup (\mathbb{D} \setminus S)$
and we can estimate the areas of these sets separately.

To estimate the area of $\mathbb{D} \setminus S$, note that the
hyperbolic area of a disk of hyperbolic radius $s$ is equal to
$4\pi \sinh^2(s/2)$, which is smaller than $4\pi s^2$ whenever
$s<4$. As $\mathbb{D}\setminus S$ is a union of $d-1$ disks of
hyperbolic radius $32d\delta<4$, we find
$$
\textrm{Area}_{\mathbb{D}}(\mathbb{D}\setminus S) \leq (d-1)4\pi
2^{10} d^2 \delta^2<2^{12} \pi d^3 \delta^2.
$$

To estimate $f|_S^{-1}(R)$, we use our estimate on the derivative.
For any open set $U \subseteq f|_S^{-1}(R)$ such that $f|_U$ is
conformal, we find
\begin{align*}
\textrm{Area}_{\mathbb{D}}(f(U))&=\iint \limits_{f(U)}\frac{4}{(1-|w|^2)^2}\,d\lambda(w)=\iint \limits_U \frac{4}{(1-|f(z)|^2)^2}|f'(z)|^2\,d\lambda(z) \\
& \geq \iint \limits_U
\frac{4}{(1-|f(z)|^2)^2}\frac{(1-|f(z)|^2)^2}{(1-|z|^2)^2}\delta^{2d-2}\,d\lambda(z)=\delta^{2d-2}
\textrm{Area}_{\mathbb{D}}(U).
\end{align*}

Since $f$ has topological degree $d$, we can therefore estimate
the Poincar\'{e} area of $f|_S^{-1}(R)$ by
$\frac{d}{\delta^{2d-2}} \cdot A=d\delta^2$. Combining this with
our estimate on the area of $S$ gives:
$$
\textrm{Area}_{\mathbb{D}}(f^{-1}(R))\leq d \delta^2+2^{12} \pi
d^3\delta^2 \leq C_3 d^3 A^{1/d}.
$$
\end{proof}

\end{document}